\documentclass{article}

\usepackage[centertags]{amsmath}
\usepackage{amsfonts}
\usepackage{amssymb}
\usepackage{amsthm}
\usepackage{newlfont}
\usepackage{mathtools}
\usepackage[top=1in, bottom=1in, left=1in, right=1in]{geometry}
\usepackage{tikz}

\theoremstyle{plain}
\newtheorem{theorem}{Theorem}[section]

\newtheorem{corollary}[theorem]{Corollary}
\newtheorem{lemma}[theorem]{Lemma}

\theoremstyle{definition}
\newtheorem{definition}[theorem]{Definition}
\newtheorem{claim}[theorem]{Claim}

\newtheorem*{remark}{Remark}

\newcommand{\ii}{\mathrm{i}}

\newcommand{\0}{\mathbf 0}
\newcommand{\wt}[1]{\widetilde{#1}}

\newcommand{\Z}{\mathbb{Z}}

\newcommand{\Hidden}[1]{}

\newcommand{\1}{\mathbf 1}

\renewcommand{\mod}{~\text{mod}~}

\begin{document}

\title{Approximate quantum fractional revival in paths and cycles}
\author{Ada Chan\footnote{Department of Mathematics and Statistics, York University, Toronto, ON, ssachan@yorku.ca}, Whitney Drazen\footnote{Department of Mathematics, Northeastern University, Boston, MA, drazen.w@husky.neu.edu}, Or Eisenberg\footnote{Department of Mathematics, Harvard University, Cambridge, MA, oreisenberg@college.harvard.edu}, Mark Kempton\footnote{Department of Mathematics, Brigham Young University, Provo, UT, mkempton@mathematics.byu.edu}, Gabor Lippner\footnote{Department of Mathematics, Northeastern University, Boston, MA, g.lippner@northeastern.edu}}

\date{}

\maketitle

\begin{abstract}
We initiate the study of approximate quantum fractional revival in graphs, 
a generalization of pretty good quantum state transfer in graphs.  We give a complete characterization of approximate fractional revival in a graph in terms of the eigenvalues and eigenvectors of the adjacency matrix of a graph.  This characterization follows from a lemma due to Kronecker on Diophantine approximation, and is similar to the spectral characterization of pretty good state transfer in graphs.  Using this, we give a complete characterizations of when approximate fractional revival can occur in paths and in cycles.  
\end{abstract}

\section{Introduction}
In recent years, tools and methods from algebraic and spectral graph theory have found important application in quantum information theory regarding the transfer of information through a quantum network.  Namely, given a graph that represents a network of interacting qubits, information transfer in this network can be modeled by a \emph{quantum walk} on the graph.  A quantum walk is a process governed by the transition matrix
\[
U(t) := e^{itA}
\]
where $A$ is the adjacency matrix of the graph. We say that there is \emph{perfect state transfer} at time $t$ between nodes $u$ and $v$ of such a network when the $(u,v)$-entry of of $U(t)$ has absolute value 1.  The study of perfect state transfer in graphs was initiated by Bose in 2003 \cite{Bose2003}, and since then, the problem has attracted considerable attention, both from the quantum information community and from the algebraic graph theory community (see for instance \cite{christandl2005,godsil,godsil2,kay2010,us,kay2011} and references therein). 

It has been shown that in simple, unweighted graphs, perfect state transfer is very difficult to achieve and occurs only rarely \cite{godsil2}.  Known constructions that achieve perfect state transfer involve either highly specialized unweighted graphs or highly non-uniform edge weights.  For instance, in simple unweighted path graphs, perfect state transfer occurs only in paths of length 2 and 3, and not in paths of any higher length \cite{godsil}.  As such, various relaxations and generalizations of perfect state transfer have been studied.  Most notably, the notion of \emph{pretty good state transfer} was introduced by Godsil et. al. \cite{Godsil2012} and Vinet et. al. \cite{Vinet2012}.  There is pretty good state transfer from node $u$ to node $v$ in a graph if for all $\epsilon>0$ there is a $t>0$ such that $|U(t)(u,v)|>1-\epsilon$.  Pretty good state transfer has been extensively studied \cite{Godsil2012,Vinet2012,PGST_old,invol,eisenberg2019pretty}.  Work in \cite{Godsil2012}, \cite{Coutinho2016}, and \cite{vanBommel2016} gives a complete characterization of when pretty good state transfer occurs in simple, unweighted paths.

Another generalization of perfect state transfer that has been studied recently is \emph{fractional revival} in graphs \cite{ChanCoutinhoTamonVinetZhan}.  For nodes $u,v$ in a graph, there is fractional revival at time $t$ from $u$ to $v$, if \[|U(t)(u,u)|^2 + |U(t)(u,v)|^2 = 1.\] 
It was shown in \cite{ChanCoutinhoTamonVinetZhan} that if there is fractional revival from $u$ to $v$, then there is fractional revival from $v$ to $u$ at the same time.  Thus, we will simply say there is fractional revival between $u$ and $v$ at time $t$.  Where prefect state transfer represents the exact transfer of a quantum state between two nodes, fractional revival represents the transfer of a quantum state to a superposition over exactly two nodes.  Fractional revival is important in the study of entanglement generation \cite{ChanCoutinhoTamonVinetZhan}.  Fractional revival in weighted paths is studied in \cite{bernard2018graph,ChanCoutinhoTamonVinetZhan,chen2007fractional,christandl2017analytic,GenestVinetZhedanov1,GenestVinetZhedanov}.
In simple, unweighted paths, it has been shown that fractional revival can occur between two nodes only in paths of length 2, 3, or 4 \cite{ChanCoutinhoTamonVinetZhan}.  

It is natural, then, to relax the condition for fractional revival to get an approximation, in the same way that perfect state transfer is relaxed to pretty good state transfer.  That is, we say there is \emph{pretty good fractional revival} between nodes $u$ and $v$ in a graph if, for all $\epsilon >0$ there is a $t>0$ such that $U(t)$ is within $\epsilon$ of exhibiting fractional revival. We formulate this more precisely in Section \ref{sec:PGFR}.  
In \cite{ChanCoutinhoTamonVinetZhan}, the study of pretty good fractional revival is listed as a major open area of research in the study of fractional revival on graphs. 

 The main results of this paper are complete characterizations of when pretty good fractional revival can occur in a simple, unweighted paths, and in simple, unweighted cycles.  Specifially, we prove the following theorems.
\begin{theorem}\label{thm:main}
Let $P_n$ denote the path on $n$ vertices, and label the vertices $1,\cdots,n$.  Then pretty good fractional revival occurs between nodes of $P_n$ if and only if we are in one of the following cases:
\begin{itemize}
\item (symmetric nodes) $n=p\cdot 2^k-1$ for $p$ a prime.  Here fractional revival occurs between nodes $a$ and $p2^k-a$ when $a$ is a multiple of $2^{k-1}$.
\item (asymmetric nodes) $n=5\cdot 2^k-1$.  Here fractional revival occurs between nodes $2^k$ and $3\cdot 2^k$ (or, by symmetry, between $2\cdot 2^k$ and $4\cdot 2^k$).
\end{itemize}
\end{theorem}

\begin{theorem}\label{thm:main2}
Let $C_n$ denote the cycle on vertices $1,\cdots, n$.  Pretty good fractional revival occurs between vertices $a$ and $b$ in $C_n$ if and only if 
$n=2p^k$, for some prime $p$, for $k\geq 1$,
and $b=a+n/2$.
\end{theorem}

In paths, we observe that between pairs of symmetric nodes, pretty good fractional revival occurs exactly when there is pretty good state transfer (see \cite{vanBommel2016}).  Surprisingly, there are instances of pretty good fractional revival between pairs of asymmetric nodes as well.  We also remark that pretty good state transfer occurs in $C_n$ if and only if $n=2^k$ for some $k$ (see \cite{MR3665556}).  Thus we extend to a larger infinite family where pretty good fractional revival occurs.  In the case of cycles, pretty good fractional revival only occurs between antipodal nodes.  

It is well-known \cite{godsil,Godsil2012} that a necessary condition for both perfect and pretty good state transfer between two nodes $u$ and $v$ is that they be \emph{strongly cospectral}, namely, that $\phi(u) = \pm\phi(v)$ for any eigenvector $\phi$ of the adjacency matrix.  In \cite{chan2020fundamentals}, we have generalized the notion of strong cospectrality to that of \emph{strong fractional cospectrality}.  That is, two vertices $u$ and $v$ are strongly fractionally cospectral if there is some constant $C\neq0$ such that either $\phi(u) = C\phi(v)$ or $\phi(u) = (-1/C)\phi(v)$ for all eigenvectors $\phi$ of the adjacency matrix.  We showed in \cite{chan2020fundamentals} that strong fractional cospectrality is a necessary condition for fractional revival.  In this paper, we show that strong fractional cospectrality is also a necessary condition for pretty good fractional revival.  In addition, work in \cite{chan2020fundamentals} characterizes fractional cospectrality in terms of a condition on walk counts at the vertices $u$ and $v$, analogous to the walk count condition for cospectrality (see \cite{godsil,godsil_smith_2017}). One of the main steps in our proof of the Theorem \ref{thm:main} is to analyze the counts of closed walks in a path to characterize when two nodes of a path are fractionally cospectral.  

In \cite{Godsil2012,invol} a characterization for pretty good state transfer between two cospectral nodes was given in terms of a number-theoretic condition on the eigenvalues of the adjacency matrix.  This number-theoretic condition arises as a consequence of a lemma due to Kronecker on Diophantine approximation.  Our primary tool in the proof of Theorems \ref{thm:main} and \ref{thm:main2}, and one of the major contributions of this paper, is a generalization of this spectral characterization for pretty good fractional revival, again relying on the Kronecker Diophantine approximation lemma (see Theorem \ref{thm:kronecker} below).

\section{Pretty good fractional revival}\label{sec:PGFR}


We remark that fractional revival at $u$ and $v$ is equivalent to saying that $U(t)$ is block diagonal at time $t$:
\[
U(t) = \begin{bmatrix}
H&0\\0&M
\end{bmatrix}
\]
where $H = \begin{bmatrix}
\alpha&\beta\\\beta&\gamma
\end{bmatrix}$ is some unitary $2\times 2$ matrix indexed by $u$ and $v$.  Perfect state transfer is the special case where $\alpha=\gamma=0$.  Pretty good fractional revival can be formulated in the same terms.  We say there is \emph{pretty good fractional revival} (PGFR) between $u$ and $v$ if for all $\epsilon>0$ there is a $t_\epsilon>0$ where $U(t_\epsilon)$ has the block form
\[
U(t_\epsilon) = \begin{bmatrix}
H(t_\epsilon) & X_\epsilon\\X_\epsilon&M_\epsilon
\end{bmatrix}
\]
 with each row of $X_\epsilon$ having magnitude at most $\epsilon$, and where $H(t_\epsilon)$ is $2\times 2$ and corresponds to vertices $u$ and $v$.
 
A necessary condition for pretty good state transfer between $u$ and $v$ is that $u$ and $v$ be strongly cospectral (see \cite{godsil}). In \cite{chan2020fundamentals}, a generalization of the notion of cospectrality is given in the framework of fractional revival.  This concept is called fractional cospectrality, which we now define.

\begin{definition}
Let $H$ be a non-diagonal $2\times2$ symmetric unitary matrix, and let $u,v$ be vertices of a graph $G$ with adjacency matrix $A$.  Given a vector $x$ indexed by vertices of $G$, let $\widetilde{x}$ be the restriction of $x$ to only vertices $u$ and $v$.
\begin{enumerate}
\item We say that $u$ and $v$ are \emph{fractionally cospectral} with respect to $H$ if there is a basis $\varphi_1,...,\varphi_n$ of eigenvectors of $A$ such that $\widetilde \varphi_i$ is either 0 or an eigenvector $H$.
\item We say that $u$ and $v$ are \emph{strongly fractionally cospectral} with respect to $H$, for every eigevector $\varphi$ of $A$, either, $\widetilde\varphi = 0$ or $\widetilde\varphi$ is an eigenvector of $H$.
\end{enumerate} 
\end{definition}

The following theorem comes from \cite{chan2020fundamentals} and gives a characterization of fractional cospectrality in terms of walk counts at $u$ and $v$ (recall that $A^k(x,y)$ enumerates walks of length $k$ between $x$ and $y$ in $G$).

\begin{theorem}[Theorem 8.3 of \cite{chan2020fundamentals}]\label{thm:frac_cosp}
Let $u$ and $v$ be vertices of $G$, and $H$ be a $2\times 2$ non-diagonal unitary symmetric matrix with eigenvectors $\psi_1 = [p,q]^T$ and $\psi_2 = [-q,p]^T$. Then the following are equivalent:
\begin{enumerate}
\item $u$ and $v$ are fractionally cospectral with respect to $H$.
\item $A^k(u,u)-A^k(v,v) = \left(\frac{p}{q} - \frac{q}{p}\right)A^k(u,v)$ for all $k$.
\end{enumerate}
\end{theorem}

This characterization allows us to investigate fractional cospectrality via the combinatorics of walks in a graph.  We will use this tool in Section \ref{sec:cosp}

Note that fractional cospectrality only requires the existence of a particular basis of eigenvectors whose restrictions are eigenvectors of $H$, while strong cospectrality requires this of all possible eigenvectors of $H$.  Thus, if all the eigenvalues are simple, then fractional cospectrality immediately implies strong fractional cospectrality.  Observe that when $H$ is not a multiple of the identity, $H$ has two distinct eigenvalues.  Fractional cospectrality with respect to $H$ naturally groups the eigenvalues of $A$ into two groups: one group is the eigenvalues for which the eigenvectors from our particular basis restrict to certain eigenvalue of $H$, and the other group is the eigenvalues for which the eigenvectors in the basis restrict to the other eigenvector of $H$.  Thus if two nodes are fractionally cosepctral but not strongly fractionally cospectral, then it must be the case that $A$ has an eigenvalue with multiplicity at least two, and this grouping puts this eigenvalue into both groups.


Notice that fractional revival includes the case that $H$ is the identity matrix, in which case the vertices $u$ and $v$ are \emph{periodic}.  However, approximate periodicity, or pretty good periodicity, by a standard compactness argument, occurs for every vertex simultaneously (see Section 6 of \cite{fan2013pretty}), and thus its study for pairs is uninteresting. Hence, in our discussion of pretty good fractional revival, we will exclude the case where two vertices are approximately periodic.

\subsection{A spectral characterization}

In this section we will give a spectral characterization of pretty good fractional revival, analogous to that for pretty good state transfer (see \cite{Godsil2012,invol,vanBommel2016}).  This characterization is based on the number theoretic lemma due to Kronecker.

\begin{lemma}[Kronecker]\label{lem:Kron}
Let $\theta_0,...,\theta_d$ and $\zeta_0,...,\zeta_d$ be arbitrary real numbers.  For an arbitrarily small $\epsilon$, the system of inequalities
\[
|\theta_ry - \zeta_r| < \epsilon ~~ (mod ~2\pi), ~~ (r=0,...,d),
\]
has a solution $y$ if and only if, for integers $\ell_0,...,\ell_d$, 
\[
\ell_0\theta_0+\cdots+\ell_d\theta_d = 0,
\]
implies
\[
\ell_0\zeta_0 + \cdots + \ell_d\zeta_d \equiv 0 ~~ (mod ~2\pi).
\]
\end{lemma}

\begin{theorem}\label{thm:kronecker}
There is pretty good fractional revival between $u$ and $v$ if and only if the following two conditions are satisfied:
\begin{enumerate}
\item $u$ and $v$ are fractionally cospectral;
\item let $\Pi_1,\Pi_2$ be the two groups of eigenvalues coming from the grouping on the eigenvalus given by the fractional cospectrality (see the comment after Theorem \ref{thm:frac_cosp}). For any integers $\ell_i$, 
\begin{align*}
\sum_{i\in\Pi_1}\ell_i\lambda_i + \sum_{j\in\Pi_2}\ell_j\lambda_j &= 0\\
\sum_i\ell_i&=0
\end{align*}
implies
\[
\sum_{i\in\Pi_1}\ell_i\neq\pm1.
\]
\end{enumerate}
\end{theorem}
We will first prove the following lemmas before proceeding with the proof of the theorem.
\begin{lemma}\label{lem:pgfr-strfrcsp}
If there is pretty good fractional revival between $u$ and $v$, then $u$ and $v$ are strongly fractionally cospectral.
\end{lemma}
\begin{proof}
Since we have PGFR from $u$ to $v$, then for every $\epsilon>0$, there is a $t_\epsilon>0$ such that we can write
\[
U(t_\epsilon) = \begin{bmatrix}
H(t_\epsilon)&X_\epsilon\\X_\epsilon^T&M_\epsilon
\end{bmatrix}
\]
where $H(t_\epsilon)$ is a $2\times 2$ matrix corresponding to $u$ and $v$, and $X_\epsilon$ is such that each row has magnitude at most $\epsilon$.  Let $H$ be a subsequential limit of $H(t_\epsilon)$ as $\epsilon\rightarrow0$.  Now let $\phi$ be a unit eigenvector of the adjacency matrix $A$, $A\phi= \lambda\phi$, and hence an eigenvector of $U(t)$ (namely $U(t)\phi = e^{it\lambda}\phi$).  Let $\widetilde\phi$ be the restriction of $\phi$ to $u$ and $v$.  Then since the rows of $X_\epsilon$ have magnitude bounded by $\epsilon$, we see that
\[
||H(t_\epsilon)\widetilde\phi - e^{it_\epsilon\lambda}\widetilde\phi ||< 2\epsilon.
\]
Letting $\epsilon\rightarrow0$ we see that 
\[
H\widetilde\phi = \rho\widetilde\phi
\]
for some subsequential limit $\rho$ of $e^{it_\epsilon\lambda}$.  Hence, either $\widetilde\phi$ is an eigenvector of $H$, or it is 0.  This shows that $u$ and $v$ are strongly fractionally cospectral.  
\end{proof}

We will now see that for fractionally cospectral vertices, condition 2. of Theorem \ref{thm:kronecker} actually implies strong fractional cospectrality. 
\begin{lemma}\label{lem:strcosp}
If Conditions 1. and 2. of Theorem \ref{thm:kronecker} are satisfied, then $u$ and $v$ are strongly fractionally cospectral.
\end{lemma}
\begin{proof}
Suppose that $u$ and $v$ are not strongly fractionally cospectral.  Then there is some eigenvalue of $A$ with multiplicity at least two that belongs to both groups in the grouping of eigenvalues.  That is $\lambda_k = \lambda_r$ for some $k\in\Pi_1,r\in\Pi_2$.  Then in Condition 2, let us choose $\ell_k=1$, $\ell_r=-1$ and $\ell_i=0$ for $i\neq k,r$.  Then clearly we have
\begin{align*}
\sum_{i\in\Pi_1}\ell_i\lambda_i + \sum_{j\in\Pi_2}\ell_j\lambda_j &= 0\\
\sum_i\ell_i&=0
\end{align*}
but 
\[
\sum_{i\in\Pi_1} \ell_i = 1
\]
contradicting Condition 2.
\end{proof}

\noindent\emph{Proof of Theorem \ref{thm:kronecker}.}
By Lemmas \ref{lem:pgfr-strfrcsp} and \ref{lem:strcosp}, we need to show that for a pair of fractionally cospectral vertices $u$ and $v$, we have pretty good fractional revival if and only if condition 2 is satisfied. 
From fractional cospectrality, we obtain a grouping of the eigenvalues and eigenvectors into two groups.  Let us denote these groups $\Pi_1$ and $\Pi_2$.  From the proof of Lemma \ref{lem:pgfr-strfrcsp}, pretty good fractional revival occurs if and only if there is a $2\times2$ symmetric unitary $H$ with eigenvalues $\rho_1,\rho_2$ such that for all $\epsilon>0$, there is a $t_\epsilon>0$ such that $e^{it_\epsilon\lambda_i}$ is close to $\rho_1$ for $i\in\Pi_1$, and $e^{it_\epsilon\lambda_j}$ is close to $\rho_2$ for $j\in\Pi_2$.  To have true pretty good fractional revival (and not approximate periodicity) we must have $\rho_1\neq\rho_2$.  This holds if and only if there is some $\delta_1,\delta_2$ with $\delta_1\not\equiv\delta_2~(mod~2\pi)$ such that the system
\begin{align*}
|t\lambda_i - \delta_1| &< \epsilon~(mod~2\pi) \text{ for } i\in\Pi_1\\
|t\lambda_i - \delta_2| &< \epsilon~(mod~2\pi) \text{ for } j\in\Pi_2
\end{align*}
has a solution $t_\epsilon$ for all $\epsilon >0$.  We need to show that this condition is equivalent to condition 2.  

Assume that we have integers $\ell_1,...,\ell_n$ with
\[
\sum_{i=1}^m\ell_i\lambda_i = 0~ \text{  and  } ~\sum_{i=1}^n\ell_i = 0
\]
By the Kronecker approximation theorem (Lemma \ref{lem:Kron}), we thus have that
\[
\sum_{i\in\Pi_1} \ell_i\delta_1 + \sum_{j\in\Pi_2}\ell_i\delta_2 = (\delta_1 - \delta_2)\sum_{i\in\Pi_1}\ell_i \equiv 0~(mod~2\pi).
\]
If 
\[
\sum_{i\in\Pi_1}\ell_i = \pm1,
\]
then this would imply $\delta_1\equiv\delta_2~(mod~2\pi)$, which would in turn imply that $\rho_1=\rho_2$.  But we are assuming pretty good fractional revival, and hence not approximate periodicity, so $\rho_1\neq\rho_2$.  Thus we conclude that 
\[
\sum_{i\in\Pi_1}\ell_i \neq \pm1.
\]

For the converse, we assume that for all $\ell_1,...,\ell_n\in\Z$,
\[
\sum_{i=1}^m\ell_i\lambda_i = 0~ \text{  and  } ~\sum_{i=1}^n\ell_i = 0
\]
implies that
\[
\sum_{i\in\Pi_1}\ell_i\neq\pm1.
\]
Define $\widetilde\lambda_i = \lambda_i - \lambda_1$
for each $i\geq2$.  Our condition is equivalent to saying
\[
\sum_{i=2}^n\ell_i\widetilde\lambda_i = 0
\]
implies
\[
\sum_{j\in\Pi_2}\ell_j \neq \pm1.
\]
Define 
\[
\mathcal S = \left\{\sum_{j\in\Pi_2}\ell_j ~\in \Z ~:~ \sum_{i=2}^n\ell_i\widetilde\lambda_i=0\right\}.
\]
Note that $\mathcal S$ is an additive subgroup of $\Z$ and is not all of $\Z$ since $\sum_{j\in\Pi_2}\ell_j$ cannot be $\pm1$.  Therefore 
\[
\mathcal S = q\Z
\]
for some integer $q\neq\pm1$.  Define 
\[
\widetilde\delta_1:=0 ~\text{ and }~ \widetilde\delta_2 := \frac{2\pi}{q} ,
\]
then
\[
\widetilde\delta_1\sum_{i\in\Pi_1,i>1}\ell_i + \widetilde\delta_2\sum_{j\in\Pi_2}\ell_j \equiv 0~(mod~2\pi)
\]
for any integers $\ell_2,...,\ell_n$ with $\sum_{i=2}^n\ell_i\widetilde\lambda_i = 0$.  Then by Lemma \ref{lem:Kron}, for all $\epsilon>0$, the system 
\begin{align*}
|t\widetilde\lambda_i - \widetilde\delta_1| &< \epsilon~(mod~2\pi) \text{ for } i\in\Pi_1\setminus\{1\}\\
|t\widetilde\lambda_i - \widetilde\delta_2| &< \epsilon~(mod~2\pi) \text{ for } j\in\Pi_2
\end{align*}
has a solution $t=t_\epsilon$.
Now choose $\delta_1$ so that $t_\epsilon\lambda_1\rightarrow\delta_1$ as $\epsilon\rightarrow0$, and set $\delta_i = \delta_1+\widetilde\delta_i$ for $i=1,2$.  Then it becomes clear that 
\begin{align*}
|t\lambda_i - \delta_1| &< \epsilon~(mod~2\pi) \text{ for } i\in\Pi_1\\
|t\lambda_i - \delta_2| &< \epsilon~(mod~2\pi) \text{ for } j\in\Pi_2
\end{align*}
has a solution for $t$, and $\delta_1\not\equiv\delta_2~(mod~2\pi)$ as desired.  This completes the proof of Theorem \ref{thm:kronecker}.
\qed

\begin{remark}
We note that condition 2 of Theorem \ref{thm:kronecker} is a generalization of the eigenvalue condition necessary for PGST, which requires that $\sum_i\ell_i$ be even.  See Lemma 1 of \cite{invol} and Lemma 3 of \cite{vanBommel2016}.
\end{remark}

\section{Paths}
\begin{theorem}
Let $P_n$ denote the path on $n$ vertices labeled $1,...,n$.  Then PGFR occurs in $P_n$ in the following cases:
\begin{enumerate}
\item[(1)] between symmetric pairs of vertices if and only if there is PGST between those points (characterized in \cite{vanBommel2016});
\item[(2)] $n=5\cdot 2^k - 1$ vertices, between vertices $2^k$ and $3\cdot2^k$ (or, by symmetry, between $2\cdot2^k$ and $4\cdot2^k$).
\end{enumerate}
\end{theorem}

For part (1) of this theorem, clearly if there is PGST between two vertices, then there is PGFR.  It remains to see that when there is not PGST between symmetric pairs of nodes, then there is not PGFR either.  It is clear that any symmetric pair of vertices in a path are strongly cospectral (see Lemma 2 of \cite{vanBommel2016}), so we need to see if the eigenvalues condition of Theorem \ref{thm:kronecker} fails when there is not PGST. By the work in \cite{vanBommel2016} where PGST in paths is characterized, in cases where there is no PGST, this is shown construction of an integer linear combination of the eigenvalues of the adjacency matrix of a path with $\sum_i\ell_i$ odd (see our remark after the proof of Theorem \ref{thm:kronecker}).  Examining the proof of Theorem 4 of \cite{vanBommel2016}, in fact, the linear combination constructed has $\sum_i\ell_i = \pm1$.  So by Theorem \ref{thm:kronecker}, \cite{vanBommel2016} has actually proven that there is no PGFR between these vertices.  

Thus we only need concern ourselves with part (2) of our theorem.  The remainder of this section is dedicated to the proof.  We will first investigate when two non-symmetric vertices of a path can be fractionally cospectral.  For these nodes, we will then investigate the eigenvalue condition of Theorem \ref{thm:kronecker}.

\subsection{Fractional cospectrality}\label{sec:cosp}

The goal of this section is to see when a pair of vertices in a path can be fractionally cospectral.  We will primarily be using the walk count characterization of fractional cospectrality from Theorem \ref{thm:frac_cosp}.  Thus, it will be helpful to have some understanding of the combinatorics of walk counts of a path.  With this in mind, we give the following lemma.

\begin{lemma}\label{lem:walk}
Let $P_n$ be the path with $n$ vertices labeled $1,...,n$.  Let $x<y$ be two vertices of $P_n$ and without loss of generality, let us assume that $x$ is closer to 1 than to $n$.  Then we have the following.
\begin{enumerate}
\item $A^k(x,y) = 0$ if $k\not\equiv y-x \mod 2$.
\item $\displaystyle A^k(x,y) = \binom{k}{\frac{k-(y-x)}{2}}-\binom{k}{\frac{k-(y+x)}{2}}$ if $k\equiv y-x\mod 2$ and $k\leq 2n-(x+y)$.
\end{enumerate}
\end{lemma}
\begin{proof}
The first item is clear since a path is bipartite.

For the second, we wish to count walks starting at $x$ and ending at $y$. Let us consider walks on a path extended to the left out to length $x-k$.  We will first count the walks in this path starting at $x$ and ending at $y$ in this extended path, and then we will subtract off the walks that used the vertex 0 at some point (so the steps to the left exceed the endpoint of the original path).  First, to go from $x$ to $y$, the number of steps to the right must exceed the number of steps to the left by exactly $y-x$.  There are $k$ steps total, so the number of steps to the left is $(k-(y-x))/2$.  Among the $k$ steps, once we have determined which steps are to the left, the walk is determined.  Thus there are $\binom{k}{(k-(y-x))/2}$ ways of choosing this.

Now we enumerate the number of these walks that use the vertex 0.  We will use a standard reflection technique coming from the combinatorics of Catalan numbers. By switching the left/right moves of the walk after the first time the walk reaches 0, the walks we wish to count are in bijective correspondence with walks starting at $x$ and ending at $-y$.  By reasoning similar to the above, we find that the number of steps to the right in such a walk must be $(k-(x+y))/2$, so as above, we obtain that the total number of such walks is $\binom{k}{(k-(y+x))/2}$.  The restriction on $k$ implies that there is no obstruction coming from the right endpoint $n$.  This gives the lemma.
\end{proof}

We remark that if $k<y-x$, there is no walk of length $k$ from $x$ to $y$, and the binomial coefficients above have negative denominator, and we interpret $\binom{k}{j}=0$ for $j<0$, so the formula is still correct.  Likewise, if $k<y+x$, then no walk of length $k$ has any chance of exceeding the endpoint, so the walk count is given by just the first term, and the amount subtracted off is 0 for the same reason.

We will record as a corollary the special case of this when $y=x$, so that we are counting closed walks.

\begin{corollary}\label{cor:closed_walk}
Let $x$ be any vertex of a path $P_n$ as above.
\begin{enumerate}
\item $A^{2r+1}(x,x)=0$ for all $r$.
\item $\displaystyle A^{2r}(x,x)=\binom{2r}{r} - \binom{2r}{r-x}$ for $r\leq n-x$.
\end{enumerate}  
\end{corollary}

Again, remark that for $r<x$, there is no obstruction from the endpoint, and the second binomial coefficient is 0.

\begin{theorem}
Non-symmetric vertices $u$ and $v$ in a path $P_n$ are strongly fractionally cospectral if and only if $n=5d-1$ for some positive integer $d$, and $\{u,v\}=\{d,3d\}$ or (by symmetry) $\{u,v\}=\{2d,4d\}$.
\end{theorem}
\begin{proof}
First, let us suppose that $u$ and $v$ are vertices of $P_n$ that are fractionally cospectral, but not cospectral (i.e. not symmetric in the path). We will repeatedly make us of the walk count characterization of Theorem \ref{thm:frac_cosp},
\begin{equation}\label{walks}
A^k(u,u)=A^k(v,v)+cA^k(u,v)
\end{equation}
for some constant $c$.  Since $u$ and $v$ are not cospectral, then  $c\neq0$.

\begin{claim}
The distance from $u$ to $v$ is even.
\end{claim}
\begin{proof}
For $k$ odd, $A^k(u,u)=A^k(v,v)=0$.  It follows form (\ref{walks}) that $A^k(u,v)=0$ for all odd $k$.  But when $k$ is the distance from $u$ to $v$, then $A^k(u,v)$ is non-zero.  Thus the distance is even.
\end{proof}

Recall that we are labeling the vertices of $P_n$ with $1,...,n$.  Without loss of generality, assume that $u<v$ and that $u$ is closer to the endpoint 1 than $v$ is to the endpoint $n$.  From the claim above, assume that the distance from $u$ to $v$ is $2d$.

\begin{claim}
Under these assumptions, it must be that $u=d$ and $v=3d$.
\end{claim}
\begin{proof}
By (\ref{walks}) we have $A^k(u,u) = A^k(v,v)$ for all $k<2d$, and we must have $A^k(u,u)\neq A^k(v,v)$ for $k=2d$.  Since $u$ is closer to an endoint of the path than $v$, then clearly, the first place $A^k(u,u)$ and $A^k(v,v)$ differ are at $k=2d(u,1)+2$ (since the endpoint closest to $u$ is 1).  Thus, this value of $k$ must be $2d$, so $u$ must be distance $d-1$ from 1, so $u=d$.  Since the distance from $u$ to $v$ is $2d$, it follows that $v=3d$.
\end{proof}

\begin{claim}\label{claim:c}
In (\ref{walks}), we must have $c=-1$.
\end{claim}
\begin{proof}
Since the distance from $u$ to 1 is $d-1$, there is one more closed walk of length $2d$ at $v$ than at $u$, and there is exactly 1 walk of length $2d$ from $u$ to $v$ since that is the distance between them.  Thus
$A^{2d}(u,u)=A^{2d}(v,v)-A^{2d}(u,v)$, and the claim follows.
\end{proof}

\begin{claim}
$n\geq5d-1$.
\end{claim}
\begin{proof}
We will look at walks of length $k=2d+2j$. We can apply Lemma \ref{lem:walk} to any $k\leq 2n-4d$, or in other words, to any $j$ with $j\leq n-3d$. Note that since $u$ is closer to $1$ than $v$ is to $n$, then $n > 4d$, thus Lemma \ref{lem:walk} applies to any $j<d$. So for any $j<d$, from Corollary \ref{cor:closed_walk} and Lemma \ref{lem:walk}, we have
\begin{align*}
A^{2d+2j}(u,u) &= \binom{2d+2j}{d+j}-\binom{2d+2j}{j}\\
A^{2d+2j}(u,v) &= \binom{2d+2j}{j}.
\end{align*}
Then from (\ref{walks}), since $c=-1$, we deduce that \[A^{2d+2j}(v,v)=\binom{2d+2j}{d+j}.\]  This implies that the distance from $v$ to the endpoint is at least $2d-1$, which implies $n\geq5d-1$, since $v=3d$.
\end{proof}

\begin{claim}
$n=5d-1$.
\end{claim}
\begin{proof}
Again using Corollary \ref{cor:closed_walk} and Lemma \ref{lem:walk}, we have
\begin{align*}
A^{4d}(u,u) &= \binom{4d}{2d} - \binom{4d}{d}\\
A^{4d}(u,v) &= \binom{4d}{2d} - \binom{4d}{0} = \binom{4d}{2d} - 1.
\end{align*}
Then (\ref{walks}) implies 
\[
A^{4d}(v,v) = \binom{4d}{d}-1.
\]
This implies, in particular, there is no path of length $2d$ from $v$ to its closest endpoint.  We already know that $v$ is at distance $3d-1$ from 1, so the distance from $v$ to $n$ is at most $2d-1$.  Thus we conclude $n=5d-1$.
\end{proof}

With these claims, we have shown that if $u$ and $v$ exhibit fractional cospectrality and are not cospectral, then this must be the situation we have claimed. 

It remains to prove the converse, namely that when $n=5d-1$, $u=d$, and $v=3d$, then $u$ and $v$ are fractionally cospectral.  We do this by verifying (\ref{walks}) for all $k$.  It suffices to verify it for $k\leq n$. Note that we have already verified this for $k\leq4d$, and that $c=-1$.  The verification for $4d<k\leq5d-1$ is a straightforward application of Lemma \ref{lem:walk} and Corollary \ref{cor:closed_walk}.  Indeed, for $j<d$,
\begin{align*}
A^{4d+2j}(u,u) &= \binom{4d+2j}{2d+j}-\binom{4d+2j}{d+j}\\
A^{4d+2j}(v,v) &= \binom{4d+2j}{2d+j}-\binom{4d+2j}{j}\\
A^{4d+2j}(u,v) &= \binom{4d+2j}{d+j} - \binom{4d+2j}{j}
\end{align*}
and (\ref{walks}) is verified, completing the proof.
\end{proof}

\subsection{Eigenvalue condition}

In this section, we wish to determine when a path on $5d-1$ vertices satisfies the second condition of Theorem \ref{thm:kronecker}.  As the eigenvalues of paths involve cosines, the following lemma from \cite{vanBommel2016} will be helpful.

\begin{lemma}[Lemma 5 of \cite{vanBommel2016}]\label{lem:cos}
Let $m$ be an odd integer, $0\leq a< k$ integers.  Then
\[
\sum_{i=0}^{m-1}(-1)^i\cos\left(\frac{(a+ik)\pi}{km}\right)=0.
\]
\end{lemma}

\begin{theorem}
A path on $5d-1$ vertices satisfies the eigenvalue condition of Theorem \ref{thm:kronecker} if and only if $d=2^k$.
\end{theorem}
\begin{proof}
Let $P_{5d-1}$ be the path with vertex set $1,...,5d-1$.  We know from the previous section that  vertices $d$ and $3d$ are strongly fractionally cospectral.  

We will first determine the grouping on the eigenvalues defined by the fractional cospectrality.  Let $x_1,...,x_{5d-1}$ be the eigenvectors for $P_{5d-1}$.  Since we already know fractional cospectrality, then we have that each $x_k$ satisfies either $x_k(d) = x_k(3d) = 0$, $x_k(d)/x_k(3d) = c$, or $x_k(d)/x_k(3d) = -1/c$ for some $c$.  We will call these three groups the ``zero group"  $Z$, and groups I and II respectively.

It is well-known (see for instance \cite{brouwer2011spectra}) that the eigenvalues of $P_{5d-1}$ are
\[
\lambda_j=2\cos\frac{\pi j}{5d},~~~~j=1,...,5d-1
\]
and the eigenvector $x_j$ for $\lambda_j$ can be given by
\[
x_j(k) = \sin\frac{\pi kj}{5d},~~~~k=1,...,5d-1
\]
for all $j$.  We are interested in the vertices $k=d$ and $k=3d$, so we have $x_j(d) = \sin(\pi j/5)$ and $x_j(3d) = \sin(3\pi j/5)$.  Both of these are 0 if $j$ is divisible by 5. Observe that if $j\equiv 1\mod 5$ or $j\equiv 4\mod 5$, then both $\sin(\pi j/5)$ and $\sin(3\pi j/5)$ are positive.  If$j\equiv 2\mod 5$ or $j\equiv 3\mod 5$, then $\sin(\pi j/5)$ is positive and $\sin(3\pi j/5)$ is negative.  Thus the three groups must be as follows:
\begin{align*}
Z &= \{\lambda_j : j\equiv 0\mod 5\}\\
I&= \{\lambda_j:j\equiv 1,4 \mod 5\}\\
II&=\{\lambda_j:j\equiv 2,3 \mod 5\}.
\end{align*}
The $Z$ group are not in the support of $d$ and $3d$ so these eigenvalues may be ignored.  Thus, the eigenvalue condition becomes that pretty good fractional revival occurs between $d$ and $3d$ if and only if
\begin{align*}
\sum_{\substack{j\\j\not\equiv0\mod 5}}\ell_j\lambda_j &= 0\\
\sum_{\substack{j\\j\not\equiv0\mod 5}}\ell_j &=0
\end{align*}
implies 
\[
\sum_{\substack{j\\j\equiv 1,4\mod 5}}\ell_j \neq \pm1.
\]
We wish to determine the values of $d$ for which this holds.  Suppose $d$ is divisible by some odd number.  Write $d=rm$ where $m$ is odd.  Then 
\[
\lambda_j = 2\cos\frac{\pi j}{5rm}.
\]
By Lemma \ref{lem:cos}, we have (for instance) that
\[
\sum_{i=0}^{m-1}(-1)^i\cos\left(\frac{(1+5ri)\pi}{5rm}\right)=0
\]
and
\[
\sum_{i=0}^{m-1}(-1)^i\cos\left(\frac{(2+5ri)\pi}{5rm}\right)=0
\]
and each term in these sums is a $\lambda_j$ with $j\equiv1\mod 5$ in the first case, and $j\equiv2\mod5$ in the second case.  Then choose the $\ell_{1+5ri}=(-1)^i$, $\ell_{2+5ri}=-(-1)^i$, and $\ell_j=0$ otherwise.  Then since $m$ is odd, we have
\begin{align*}
\sum_j\ell_j\lambda_j = \sum_{i=0}^{m-1}(-1)^i\cos\left(\frac{(1+5ri)\pi}{5rm}\right)-\sum_{i=0}^{m-1}(-1)^i\cos\left(\frac{(2+5ri)\pi}{5rm}\right)&=0\\
\sum_j\ell_j =\sum_{i=0}^{m-1}(-1)^i - \sum_{i=0}^{m-1}(-1)^i &=0\\
\sum_{\substack{j\\j\equiv1,4\mod5}}\ell_j = \sum_{i=0}^{m-1}(-1)^i &=1.
\end{align*}
Thus when $d$ is divisible by an odd number, we do not get pretty good fractional revival.

It remains to show that when $d$ is a power of 2, then we do get pretty good fractional revival between vertices $u=d$ and $v=3d$.  Let us write $d=2^k$, so $n=5\cdot 2^k-1$.  Using Theorem \ref{thm:kronecker} again, let us suppose that
\[
\sum_{j=1}^{5d-1}\ell_j2\cos\frac{\pi j}{5d}=0
\]
with $\ell_j\in\Z$. As we saw above, those $j$ divisible by 5 correspond to eigenvectors not supported on $u$ and $v$, so we may assume that $\ell_{5i}=0$ for all $i$.  We can rewrite this as
\[
\sum_{j=1}^{5d-1}\ell_j(\zeta^j+\zeta^{10d-j})=0
\]
where $\zeta = e^{2\pi i/10d}$ is a primitive $10d$th root of unity.  Let us alternatively write this as
\[
\sum_{j=1}^{10d-1}\ell_j\zeta^j
\]
with the condition
\begin{equation}\label{ell}
\ell_j = \ell_{10d-j}
\end{equation}
for $j=1,...,5d-1$.  Since we are still assuming $\ell_j=0$ for $j$ divisible by 5, we have that
\[
\sum_{j\equiv0\mod5}\ell_j\zeta^j=0.
\]  Thus by Theorem 2.3 of \cite{lenstra} we have that
\[
\sum_{j\equiv r\mod5}\ell_j\zeta^j=0
\]
for all $r$.  In particular, we have
\[
\sum_{j\equiv1\mod5}\ell_j\zeta^j=0.
\]
Let us define the polynomial 
\[
P(x) = \sum_{\substack{j\equiv1\mod5\\1\leq j<10d}}\ell_jx^{(j-1)/5}.
\]
Then $P$ has $x=\zeta^5$ as a root.  Note that $\zeta^5$ is a $2^{k+1}$th root of unity, and thus $P(x)$ is divisible by the $2^{k+1}$th cyclotomic polynomial
\[
\Phi_{2^{k+1}}(x) = x^{2^k}+1.
\]
This, in particular, implies that $P(1)$ is even.  But
\begin{align*}
P(1) &= \sum_{\substack{j\equiv1\mod5\\1\leq j<10d}}\ell_j\\ &= \sum_{\substack{j\equiv1\mod5\\1\leq j<5d}}\ell_j+\sum_{\substack{j\equiv1\mod5\\5d+1\leq j<10d}}\ell_j\\
&=\sum_{\substack{j\equiv1\mod5\\1\leq j<5d}}\ell_j+\sum_{\substack{j\equiv4\mod5\\1\leq j<5d}}\ell_{10d-j}\\
&=\sum_{\substack{j\equiv1,4\mod5\\1\leq j<5d}}\ell_j
\end{align*}
where the last line follows from (\ref{ell}).  Since this sum is even, it is not $\pm1$, and thus we get pretty good fractional revival by Theorem \ref{thm:kronecker}.
\end{proof}

\section{Cycles}

In this section, we extend Pal and Bhattacharjya's characterization of PGST in cycles, \cite{MR3665556}, to PGFR.
They showed that such cycles must have $2^k$ vertices, for some $k\geq 2$, and PGST must occur between antipodal vertices.
Since PGST is a special case of PGFR, these cycles have PGFR between their antipodal vertices.
We determine all the cycles admitting PGFR.

Let $C_n$ be a cycle on vertex set $\Z_n$ with vertex $a$ adjacent to $a-1$ and $a+1$, for $a\in \Z_n$.  
Let $\omega$ be a primitive $n$-th root of unity.   Then the adjacency matrix, $A$, of $C_n$
has eigenvectors 
\begin{equation*}
\phi_j=\begin{bmatrix} w^j & w^{2j} & \cdots& w^{(n-1)j} \end{bmatrix}^T
\end{equation*}
with corresponding eigenvalue $\lambda_j = 2\cos \frac{2\pi j}{n}$, for $j = 0, 1\ldots,n-1$.
For each $j$,
\begin{equation*}
\phi_j(b) = \omega^{(b-a)j} \phi_j(a).
\end{equation*}
Hence there exists non-zero constant $c$ such that $\omega^{(a-b)j}$ is equal to $c$ or $-1/c$, for all $j$, if and only if 
$n$ is even, $b=a+n/2$, and $c=1$.   In this case, we have
\begin{equation}
\label{Eqn:Cycle}
\phi_j\left(a+\frac{n}{2}\right) = 
\begin{cases}
 \phi_j(a) & \text{ if $j$ is even,}\\
 -\phi_j(a) &\text{ if $j$ is odd.}
 \end{cases}
\end{equation}
In other words, antipodal vertices in even cycles are the only strongly (fractionally) cospectral vertices in cycles.

We first show that there is PGFR in cycles on $2p^k$ vertices, for any odd prime $p$ and $k\geq 1$.
\begin{lemma}
Let $p$ be an odd prime and $n=2p^k$ for some $k\geq 1$.  Then $C_n$ has PGFR between antipodal vertices.
\end{lemma}
\begin{proof}
Let $\ell_0, \ell_1,\ldots, \ell_{n-1}$ be integers satisfying
\begin{equation}
\label{Eqn:Cycle2pk1}
\sum_{r=0}^{n} \ell_r \lambda_r=0
\end{equation}
and
\begin{equation}
\label{Eqn:Cycle2pk2}
\sum_{r=0}^{n} \ell_r = \sum_{r\ \mathrm{even}} \ell_r + \sum_{s\ \mathrm{odd}} \ell_s =0.
\end{equation}
Since
\begin{equation*}
\lambda_r= \lambda_{2p^k-r}=-\lambda_{p^k-r}=-\lambda_{p^k+r}, \quad \text{for  $r=1,\ldots, (p^k-1)/2$,}
\end{equation*}
we can rewrite Equation~(\ref{Eqn:Cycle2pk1}) as
\begin{equation*}
2 \left(\ell_0 - \ell_{p^k}\right) + \sum_{r=1}^{(p^k-1)/2} \left(\ell_r+\ell_{2p^k-r}-\ell_{p^k-r}-\ell_{p^k+r}\right) \lambda_r = 0.
\end{equation*}
Define 
\begin{equation*}
h_r =
\begin{cases}
(-1)^r \left(\ell_r+\ell_{2p^k-r}-\ell_{p^k-r}-\ell_{p^k+r}\right) & \text{if $r=1,\ldots, \frac{p^k-1}{2}$,}\\
\ell_0-\ell_{p^k} & \text{if $r=0$.}
\end{cases}
\end{equation*}
Then
\begin{equation*}
\sum_{r=0}^{(p^k-1)/2} h_r =  \sum_{r\ \mathrm{even}} \ell_r - \sum_{s\ \mathrm{odd}} \ell_s,
\end{equation*}
and Equation~(\ref{Eqn:Cycle2pk1}) yields
\begin{equation*}
2h_0 + \sum_{r=1}^{(p^k-1)/2} h_r \left( (-\omega)^r + (-\omega)^{-r} \right)= 0.
\end{equation*}
As a result, $-\omega$ is a root of the polynomial
\begin{equation*}
h(x) = \sum_{r=1}^{(p^k-1)/2} h_r x^{\frac{p^k-1}{2}+r} + 2h_0 x^{\frac{p^k-1}{2}} +  \sum_{r=1}^{(p^k-1)/2} h_r x^{\frac{p^k-1}{2}-r}.
\end{equation*}
The $2p^k$-th cyclotomic polynomial satisfies $\Phi_{2p^k}(x) = \Phi_{p^k}(-x)$.
Hence $-\omega$ is a root of 
\begin{equation*}
\Phi_{p^k}(x) = \sum_{s=0}^{p-1} x^{p^{k-1}s}.
\end{equation*}
As a result, there exists a unique polynomial $g(x)$ of degree $p^{k-1}-1$ such that
\begin{equation*}
h(x)=g(x) \Phi_{p^k}(x).
\end{equation*}
When $x=1$, we get
\begin{equation*}
h(1)=2\sum_{r=0}^{(p^k-1)/2} h_r =2\sum_{r\ \mathrm{even}} \ell_r - 2\sum_{s\ \mathrm{odd}} \ell_s  =  g(1) p.
\end{equation*}
Together with Equation~(\ref{Eqn:Cycle2pk2}), we get
\begin{equation*}
4 \sum_{r\ \mathrm{even}} \ell_r   =  g(1) p.
\end{equation*}
We conclude that $\sum_{r\ \mathrm{even}} \ell_r \neq \pm1$ and there is PGFR between antipodal vertices in $C_{2p^k}$.
\end{proof}

We need to the following equation from \cite[Equation (7)]{MR3665556} to rule out PGFR in $C_n$ when $n$ is divisible by two distinct odd primes.
Let $n=mp$ for some odd prime $p$.  Then
\begin{equation}
\label{Eqn:EigenvaluesCn}
\left(\lambda_2 - \lambda_1\right) + \sum_{r=1}^{\frac{p-1}{2}}\left( \lambda_{mr+2}- \lambda_{mr+1}\right) - \sum_{r=1}^{\frac{p-1}{2}}\left( \lambda_{mr-1}- \lambda_{mr-2}\right)=0.
\end{equation}

\begin{lemma}
If $n$ is divisible by $2pq$ for some distinct odd primes $p$ and $q$, then there is no PGFR in $C_n$.
\end{lemma}
\begin{proof}
Suppose $C_n$ has PGFR between antipodal vertices.
There exists a $2\times 2$ non-diagonal symmetric unitary matrix 
\begin{equation*}
H = \begin{bmatrix} \alpha & \beta\\ \beta & \gamma \end{bmatrix}
\end{equation*}
such that, for each $\epsilon >0$, there exists $t_{\epsilon}$ with
\begin{equation*}
U(t_{\epsilon}) \approx 
\begin{bmatrix} H & \0\\ \0 & H' \end{bmatrix}.
\end{equation*}
By Equation~(\ref{Eqn:Cycle}), we have
\begin{equation*}
\wt{\phi}_j = 
\begin{cases}
\begin{bmatrix}1 &1\end{bmatrix}^T & \text{if $j$ is even,}\\
\begin{bmatrix}1 &-1\end{bmatrix}^T & \text{if $j$ is odd,}
\end{cases}
\end{equation*}
and $\wt{\phi}_j$ being an eigenvector of $H$ implies $\gamma=\alpha$.
In this case, we have 
\begin{equation*}
H \wt{\phi}_j= (\alpha+(-1)^j \beta) \wt{\phi}_j  \quad \text{for $j=0,\ldots,n-1$.}
\end{equation*}
Since $H$ is a non-diagonal unitary matrix, we have $\beta\neq 0$ and $\alpha \neq \pm \beta$.

Let $\mu$ be the non-zero real number satisfying $e^{\ii \mu} = (\alpha+\beta)/(\alpha-\beta)$.
It follows from $U(t_{\epsilon})\phi_j = e^{\ii t_{\epsilon} \lambda_j} \phi_j$ that
\begin{equation*}
t_{\epsilon} (\lambda_j - \lambda_{j-1}) \approx
\begin{cases}
\mu \pmod{2\pi} & \text{if $j$ is even,}\\
-\mu \pmod{2\pi} & \text{if $j$ is odd.}\\
\end{cases}
\end{equation*}
Applying Lemma~\ref{lem:Kron} to Equation~(\ref{Eqn:EigenvaluesCn}) yields 
\begin{equation*}
p\mu = 0 \pmod{2\pi}.
\end{equation*}
Similarly, we have $q\mu=0\pmod{2\pi}$.

As $p$ and $q$ are coprime, there exist integers $h$ and $k$ such that $hp+kq=1$.
Then $\mu = hp\mu+kq\mu=0\pmod{2\pi}$, which implies $\beta=0$, a contradiction to $H$ being non-diagonal.
\end{proof}

\begin{lemma}
Let $n=2^hp^s$ for some odd prime $p$, $s\geq 1$ and $h\geq 2$.  There is no PGFR in $C_n$.
\end{lemma}
\begin{proof}
Applying Lemma~\ref{lem:cos} with $m=p^s$ and $k=2^{h-1}$, we get
\begin{equation*}
\sum_{j=0}^{p^s-1} (-1)^j \lambda_{2^{h-1}j+a}=0, \quad \text{for $0\leq a <2^{h-1}$.}
\end{equation*} 
When $a=0$ and $a=1$, we get 
\begin{equation*}
\sum_{j=0}^{p^s-1} (-1)^j  \lambda_{2^{h-1}j}=0
\quad \text{and}\quad
\sum_{j=0}^{p^s-1} (-1)^j  \lambda_{2^{h-1}j+1}=0
\end{equation*}
Define
\begin{equation*}
\ell_r = 
\begin{cases}
(-1)^j, & \text{if $r=2^{h-1}j$,}\\
(-1)^{j+1}, & \text{if $r=2^{h-1}j+1$,}\\
0 &\text{otherwise.}
\end{cases}
\end{equation*}
Then we have $\sum_{r=0}^{n-1}\ell_r\lambda_r=0$, $\sum_{r=0}^{n-1} \ell_r = 0$ but
\begin{equation*}
 \sum_{r\ \mathrm{even}} \ell_r =1.
\end{equation*}
Hence there is no PGFR in $C_n$.
\end{proof}

\begin{theorem}
Pretty good fractional revival occurs between $a$ and $b$ in $C_n$ if and only if 
$n=2p^k$, for some prime $p$, for $k\geq 1$,
and $b=a+n/2$.
\qed
\end{theorem}


\paragraph{Acknowledgements} G.L. was supported by NSF/DMS-1800738 and the Simons Foundation Collaboration Grant. O.E was supported by the Herchel Smith Harvard Undergraduate Research Program. 

We all acknowledge the support of Chris Godsil's NSERC Accelerator Grant that funded the workshop Algebraic Graph Theory and Quantum Walks, where we all started to work on this project.

\bibliographystyle{plain}
\bibliography{quantum}

\begin{thebibliography}{10}

\bibitem{PGST_old}
L.~Banchi, G.~Coutinho, C.~Godsil, and S.~Severini.
\newblock Pretty good state transfer in qubit chains--the heisenberg
  hamiltonian.

\bibitem{bernard2018graph}
Pierre-Antoine Bernard, Ada Chan, {\'E}rika Loranger, Christino Tamon, and Luc
  Vinet.
\newblock A graph with fractional revival.
\newblock {\em Physics Letters A}, 382(5):259--264, 2018.

\bibitem{Bose2003}
S.~Bose.
\newblock Quantum communication through an unmodulated spin chain.
\newblock {\em Physical Review Letters}, 91(20):207901, 2003.

\bibitem{brouwer2011spectra}
Andries~E Brouwer and Willem~H Haemers.
\newblock {\em Spectra of graphs}.
\newblock Springer Science \& Business Media, 2011.

\bibitem{chan2020fundamentals}
Ada Chan, Gabriel Coutinho, Whitney Drazen, Or~Eisenberg, Chris Godsil, Gabor
  Lippner, Mark Kempton, Christino Tamon, and Hanmeng Zhan.
\newblock Fundamentals of fractional revival in graphs.
\newblock {\em arXiv preprint arXiv:2004.01129}, 2020.

\bibitem{ChanCoutinhoTamonVinetZhan}
Ada Chan, Gabriel Coutinho, Christino Tamon, Luc Vinet, and Hanmeng Zhan.
\newblock Quantum fractional revival on graphs.
\newblock {\em Discrete Applied Mathematics}, 269:86--98, 2019.

\bibitem{chen2007fractional}
Bing Chen, Z~Song, and CP~Sun.
\newblock Fractional revivals of the quantum state in a tight-binding chain.
\newblock {\em Physical Review A}, 75(1):012113, 2007.

\bibitem{christandl2005}
Matthias Christandl, Nilanjana Datta, Tony~C. Dorlas, Artur Ekert, Alastair
  Kay, and Andrew~J. Landahl.
\newblock Perfect transfer of arbitrary states in quantum spin networks.
\newblock {\em Phys. Rev. A}, 71:032312, Mar 2005.

\bibitem{christandl2017analytic}
Matthias Christandl, Luc Vinet, and Alexei Zhedanov.
\newblock Analytic next-to-nearest-neighbor {XX} models with perfect state
  transfer and fractional revival.
\newblock {\em Physical Review A}, 96(3):032335, 2017.

\bibitem{Coutinho2016}
Gabriel Coutinho, Krystal Guo, and Christopher~M. van Bommel.
\newblock Pretty good state transfer between internal nodes of paths.
\newblock {\em Quantum Inf. Comput.}, 17(9-10):825--830, 2017.

\bibitem{eisenberg2019pretty}
Or~Eisenberg, Mark Kempton, and Gabor Lippner.
\newblock Pretty good quantum state transfer in asymmetric graphs via
  potential.
\newblock {\em Discrete Mathematics}, 342(10):2821--2833, 2019.

\bibitem{fan2013pretty}
Xiaoxia Fan and Chris Godsil.
\newblock Pretty good state transfer on double stars.
\newblock {\em Linear Algebra and Its Applications}, 438(5):2346--2358, 2013.

\bibitem{GenestVinetZhedanov1}
Vincent~X. Genest, Luc Vinet, and Alexei Zhedanov.
\newblock Exact fractional revival in spin chains.
\newblock {\em Modern Phys. Lett. B}, 30(26):1650315, 7, 2016.

\bibitem{GenestVinetZhedanov}
Vincent~X. Genest, Luc Vinet, and Alexei Zhedanov.
\newblock {Quantum spin chains with fractional revival}.
\newblock {\em Annals of Physics}, 371:348--367, aug 2016.

\bibitem{godsil}
C.~Godsil.
\newblock State transfer on graphs.
\newblock {\em Discrete Math}, 312(1):129--147, 2012.

\bibitem{godsil2}
C.~Godsil.
\newblock When can perfect state transfer occur?
\newblock {\em Electronic J. Linear Algebra}, 23:877--890, 2012.

\bibitem{Godsil2012}
C.~Godsil, S.~Kirkland, S.~Severini, and J.~Smith.
\newblock Number-theoretic nature of communication in quantum spin systems.
\newblock {\em Physical Review Letters}, 109(5):050502, 2012.

\bibitem{godsil_smith_2017}
Chris Godsil and Jamie Smith.
\newblock Strongly cospectral vertices.
\newblock 2017.

\bibitem{kay2010}
A.~Kay.
\newblock Perfect, efficient, state transfer and its applications as a
  constructive tool.
\newblock {\em Int. J. Quantum Inform}, 8(4):641, 2010.

\bibitem{us}
M.~Kempton, G.~Lippner, and S.-T. Yau.
\newblock Perfect state transfer on graphs with a potential.
\newblock {\em Quant. Inf. Comput}, 17(3):303--327, 2017.

\bibitem{invol}
Mark Kempton, Gabor Lippner, and Shing-Tung Yau.
\newblock Pretty good quantum state transfer in symmetric spin networks via
  magnetic field.
\newblock {\em Quantum Inf. Process.}, 16(9):16:210, 2017.

\bibitem{lenstra}
H.~W. Lenstra, Jr.
\newblock Vanishing sums of roots of unity.
\newblock In {\em Proceedings, {B}icentennial {C}ongress {W}iskundig
  {G}enootschap ({V}rije {U}niv., {A}msterdam, 1978), {P}art {II}}, volume 101
  of {\em Math. Centre Tracts}, pages 249--268. Math. Centrum, Amsterdam, 1979.

\bibitem{MR3665556}
Hiranmoy Pal and Bikash Bhattacharjya.
\newblock Pretty good state transfer on circulant graphs.
\newblock {\em Electronic Journal of Combinatorics}, 24(2):Paper 2.23, 13,
  2017.

\bibitem{kay2011}
Peter~J. Pemberton-Ross and Alastair Kay.
\newblock Perfect quantum routing in regular spin networks.
\newblock {\em Phys. Rev. Lett.}, 106:020503, Jan 2011.

\bibitem{vanBommel2016}
Christopher~M. van Bommel.
\newblock A complete characterization of pretty good state transfer on paths.
\newblock {\em Quantum Inf. Comput.}, 19(7-8):601--608, 2019.

\bibitem{Vinet2012}
Luc Vinet and Alexei Zhedanov.
\newblock Almost perfect state transfer in quantum spin chains.
\newblock {\em Phys. Rev. A}, 86:052319, Nov 2012.

\end{thebibliography}

\end{document}